\documentclass[a4paper,reqno,11pt]{amsart}
\usepackage{amsfonts,amssymb,verbatim,amsmath,amsthm,latexsym,textcomp,amscd}
\usepackage{latexsym,amsfonts,amssymb,epsfig,verbatim}
\usepackage{amsmath,amsthm,amssymb,latexsym,graphics,textcomp}
\usepackage{graphicx}
\usepackage{color}
\usepackage{url}
\usepackage{enumerate}
\usepackage[mathscr]{euscript}
\usepackage{mathtools} 
\usepackage{tikz}
\usetikzlibrary{cd}
\usepackage{dsfont}
\usetikzlibrary{matrix}
\usepackage{sseq}
\usepackage{hyperref}

\input xy
\xyoption{all}

\theoremstyle{definition}
\newtheorem{theorem}{Theorem}[section]
\newtheorem{prop}[theorem]{Proposition}

\newtheorem{cor}[theorem]{Corollary}
\newtheorem{defn}[theorem]{Definition}
\newtheorem{conj}[theorem]{Conjecture}
\newtheorem{rmk}[theorem]{Remark}

\newtheorem*{theorema}{Theorem A}
\newtheorem*{theoremb}{Theorem B}
\newtheorem*{theoremc}{Theorem C}

\usepackage{tipa}

\newcommand\FF{{\mathcal F}}

\newcommand\LL{{\mathcal L}}
\newcommand\MM{{\mathcal M}}

\newcommand\PP{{\mathcal P}}

\newcommand\PMF{{\PP\kern-2pt\MM\FF}}

\newcommand\PML{{\PP\kern-2pt\MM\LL}}

\newcommand{\fsubd}{\mathrel{{\scriptstyle\searrow}\kern-1ex^d\kern0.5ex}}
\newcommand{\bsubd}{\mathrel{{\scriptstyle\swarrow}\kern-1.6ex^d\kern0.8ex}}
\newcommand{\fsubeq}{\mathrel{\raise-.7ex\hbox{$\overset{\searrow}{=}$}}}
\newcommand{\bsubeq}{\mathrel{\raise-.7ex\hbox{$\overset{\swarrow}{=}$}}}

\newcommand{\tsh}[1]{\left\{\kern-.9ex\left\{#1\right\}\kern-.9ex\right\}}

\numberwithin{equation}{section}
\newenvironment{myeq}[1][]
{\stepcounter{theorem}\begin{equation}\tag{\thetheorem}{#1}}
	{\end{equation}}
\newenvironment{meq*}[1][]
{\stepcounter{theorem}\begin{equation*}\tag{\thetheorem}{#1}}
	{\end{equation*}}
\newtheorem{subsec}[theorem]{}

\newcommand{\hK}{\hat{K}}
\newcommand{\hH}{\hat{H}}
\usepackage{booktabs} % For better table formatting
\usepackage{amsmath}  % For mathematical symbols
\usepackage{array}    % For better column control
\newcolumntype{C}{>{\centering\arraybackslash}m{1.5cm}} % Centered column type

\begin{document}

\title[Coloring discrete pseudomanifolds]{Coloring discrete pseudomanifolds}

\author{Biplab Basak}
\address{Math Department\\ Indian Institute of Technology Delhi\\ New Delhi 110016\\ India}
\email{biplab@iitd.ac.in}
% %
\author{Vanny Doem}
\address{Math Department\\ Indian Institute of Technology Kharagpur\\ Kharagpur 721302\\ India}
\email{vanny.doem@gmail.com}
% %
\author{Chandal Nahak}
\address{Math Department\\ Indian Institute of Technology Kharagpur\\ Kharagpur 721302\\ India}
\email{cnahak@maths.iitkgp.ac.in}
% %

\date{\today}

\maketitle

\begin{abstract}
\leftskip=0cm \rightskip=0cm
This paper presents three main results on coloring discrete $d$-pseudomanifolds: $(1)$ the general chromatic bounds $d+1 \le X(K) \le 2d+2$ for any $d$-pseudomanifold $K$; $(2)$ an improved bound $X(K) \le 2d+1$ for pseudomanifolds expressible as a Zykov join $K = S^k + K'$; $(3)$ the optimal bound $X(K)\leq\lceil 3(d+1)/2\rceil$ under the additional assumptions that the spherical join factor $S^k$ is built from even-cycles and its dimension $k$ is close to $d$.
\end{abstract}

%\tableofcontents

\section{Introduction}
The chromatic theory of geometric graphs has its profound modern origin in the combinatorial topology pioneered by Steve Fisk in a seminal series of works \cite{Fis73, Fis1973, Fis74, Fis77}. Fisk investigated colorings of triangulated manifolds and demonstrated how purely combinatorial arguments could be used to solve geometric problems. However his framework was inherently tied to classical simplicial complexes and their topological realizations.

The translation of this paradigm into the pure language of finite graph theory was achieved by Oliver Knill through the theory of discrete manifolds \cite{Kni21}. By defining a discrete manifold $M$ purely as a finite simple graph where every unit sphere, the subgraph generated by the neighbors of $x$, is a $(d-1)$-sphere, Knill provided foundational bounds of chromatic number $X(M)$ for any discrete $d$-manifold $M$, \cite[Theorem 1]{Kni21},
\begin{myeq}\label{eq:bound knill}
    d+1\leq X(M)\leq 2d+2.
\end{myeq}

Furthermore using the Zykov join for two general finite graphs $G$ and $H$ (Definition \ref{def:zykov join}), where the chromatic number $X(G+H)=X(G)+X(H)$ \cite[\S3.1]{Kni21}, Knill constructed high-dimensional spheres with chromatic number scaling as $3k$ or $3k+1$ in \cite[\S3.2]{Kni21}. This leads to the conjecture that the optimal bounds might be tighter as \cite[\S1.2]{Kni21}
\begin{myeq}\label{intro:bound 2d+1}
    d+1\leq X(M)\leq 2d+1
\end{myeq}and
\begin{myeq}\label{intro:bound 1.5d}
    d+1\leq X(M)\leq\lceil 3(d+1)/2\rceil
\end{myeq}where $\lceil - \rceil$ is a ceiling function. Therefore Knill's work can be viewed as a graph-theoretic re-foundation and extension of the geometric coloring problems initiated by Fisk in \cite{Fis77}.

This article extends the chromatic theory from discrete $d$-manifolds to the significantly broader class of discrete $d$-pseudomanifolds (Definition \ref{def:pman}). A central question is whether the foundational bounds in \eqref{eq:bound knill} persist for this more general class. We affirm that the chromatic number of any discrete $d$-pseudomanifold $K$ satisfies the same robust bounds in \eqref{eq:bound knill} (Theorem \ref{thm:color pman}).
\begin{theorema}
  The chromatic number $X(K)$ of any $d$-pseudomanifold $K$ satisfies
   $$d + 1 \leq X(K) \leq 2d + 2.$$ 
\end{theorema}

 Follow that the chromatic number $X(S^k)$ can be relatively small (Proposition \ref{prop:color sphere min}) and construct a discrete $d$-pseudomanifold $K$ in such a way that it can be decomposed as $K=S^k+K'$. We then obtain the tighter bounds in \eqref{intro:bound 2d+1} (Theorem \ref{thm:bound 2d+1}).
\begin{theoremb}
If $K=S^k+K'$ such that $S^k$ is a $k$-sphere and $K'$ is a $(d-k-1)$-pseudomanifold,
$$d+1\leq X(K)\leq 2d+1.$$
\end{theoremb}

Furthermore we demonstrate that under optimal conditions on the spherical factor $S^k$, say $S^k$ is built from even cycles and is of sufficiently high dimension, the conjectured optimal ceiling in \eqref{intro:bound 1.5d} is attainable (Theorem \ref{thm:bound ceiling 1.5(d+1)}).
\begin{theoremc}
If $K=S^k+K'$ such that $S^k$ is an even-cycle $k$-sphere, $K'$ is a $(d-k-1)$-pseudomanifold, and $k$ is close enough to $d$,
$$d+1\leq X(K)\leq \lceil 3(d+1)/2\rceil.$$
\end{theoremc}

\subsection{Organization}
In \S\ref{sec:prelim} we recall some preliminaries that are used throughout the paper.  The \S\ref{sec:chrom pseudoman} deals with the case of chromatic number of a $d$-pseudomanifold proving Theorem A, and \S\ref{sec:chrom arith} deals with the foundational results for proving Theorems B and C in \S\ref{sec:sharper bounds}. 
\subsection{Notation} Throughout this paper we use the following notations. 
\begin{itemize}
\item $G$ and $H$ denote finite graphs.
\item $X(-)$ is for the chromatic number.
\item $K$ stands for a discrete pseudomanifold, and $M$ for a discrete manifold. For convenience we sometimes omit the prefix `discrete'.
\end{itemize}

\noindent
{\bf Acknowledgments.}
The first author (Biplab Basak) was funded by the Mathematical Research Impact Centric Support (MATRICS) Research Grant (No. MTR/2022/000036), SERB, and the second by the Indian Council for Cultural Relations (ICCR), Research Grant (No. PHN/ICCR/321/1/2021), India.
\section{Preliminaries}\label{sec:prelim}
\subsection{Dualities in graphs}
Let $G=(V,E)$ be a graph with vertices $V$ and edges $E$. We then recall two types of dualities in graph-theoretical notions as follows.
\begin{defn}\label{def:compl dual}
    Let $H$ be a subgraph of a finite simple graph $G$. Define a complementary dual graph $\hat{H}$ of a subgraph $H$ in $G$ to be the intersection of all unit links
    $$\hat{H}=\bigcap_{y\in H} L(y).$$
\end{defn}
For examples, the complementary dual graph of the empty graph $0$ in $G$ is the $G$ itself (and vice versa), the complementary dual graph of a $1$-vertex subgraph $x$ in $G$ is the unit link $L(x)$ in $G$, and the complementary dual graph of a complete graph $K_n$ in $K_{m}$ is the $K_{m-n}$, formed by the complementary vertices of $K_n$. 
\begin{defn}
    A dual graph $G^* = (V^*, E^*)$ of a $d$-graph $G$ is defined to have the maximal $d$-simplices as vertices and connects two different maximal simplices if they intersect in a $(d-1)$-simplex.
\end{defn}
For examples, the dual graph of $W_n$ a wheel graph with a cyclic graph $C_n$ as its boundary is the $C_n$, the dual graph of the octahedron graph ($2$-sphere) is the cube graph, and the dual graph of the $3$-sphere is the tesseract. Notice that when we say spheres, we mean topological discrete spheres in the theory of discrete manifolds in \cite[\S1.1]{Kni21}.

%One notes that: The dual graph of a $d$-pseudomanifold is always zero or one-dimensional, but for any connected $d$-pseudomanifold of dimension $d\geq 2$ different from a simplex, it is always one-dimensional. Also, the dual graph $K^*$ of a $d$-pseudomanifold $K$ is always $(d + 1)$-regular because every simplex is connected with exactly $d + 1$ neighbors.

\subsection{Arithmetic on graphs}
In this subsection we recall two algebraic operations on graphs: The Zykov join first introduced by Zykov in \cite{Zyk52}, and the Cartesian simplex product by Stanley-Reisner \cite{Kni17}.
\begin{defn}\label{def:zykov join}
Let $G=(V,E)$ and $H=(W,F)$ be finite simple graphs. Define the Zykov join $G+H$ to be the disjoint union $G \cup H$ where every vertex in $G$ is connected to every vertex in $H$. Algebraically, 
$$G+H=(V \cup W, E\cup F \cup VW)$$
where $VW$ denotes all the edges connecting any vertex $V$ to any vertex $W$ as a sum. 
\end{defn}

\begin{defn}\label{def:cartesian pro}
   For two graphs $G$ and $H$, the Cartesian simplex product $(G \times H)_1$ is the graph with the vertices as the pairs $(g, h)$, where $g$ and $h$ are complete subgraphs of $G$ or $H$ and $(g, h)$ and $(u, v)$ are connected if they are different and either $g \subset u, h \subset v$ or $u \subset g, v \subset h$. Algebraically, 
   $$(G\times H)_1 = (c(G)\times c(H), \{(a,b) ~|~ a\subset b \text{ or } b\subset a\}).$$  
\end{defn}
If $H=1$ is the one-point graph, the product $(G \times 1)_1$ is the Barycentric refinement of $G$.

\section{Chromatic Numbers of Pseudomanifolds}\label{sec:chrom pseudoman}
There is a geometric construction of topological discrete spheres, which leads to a new formulation of a topological discrete manifold in terms of graph-theoretical notions \cite[\S1.1]{Kni21}. In this section we define a new class of $d$-graphs, which we call a discrete $d$-pseudomanifold. For conventional purpose we sometimes omit the prefix `discrete'. 

\begin{defn}\label{def:pman}
A discrete $d$-pseudomanifold is a finite simple graph $K$ where every unit link $L(x)$, the subgraph generated by the neighbors of $x$, is a $(d-1)$-pseudomanifold. It is primed with the presumption that the empty graph $\emptyset$ is a $(-1)$-link, and a $1$-pseudomanifold is a cyclic graph $C_n$ for $n\geq 4$.
\end{defn}
In Definition \ref{def:pman} notice that a unit link $L(x)$ is essentially a pseudomanifold, which is not necessarily a topological sphere, as introduced in \cite[\S1.1]{Kni21} and \cite[\S8]{Kni14}. Therefore a pseudomanifold is, in general, not a topological manifold anymore.
\begin{defn}\label{def:sub pman}
    A subpseudomanifold $K'$ of a pseudomanifold $K$ is also a pseudomanifold whose vertex and edge sets are subsets of the vertex and edge sets of the pseudomanifold.
\end{defn}
Following Definitions \ref{def:compl dual} and \ref{def:pman}, we generate the following facts.
\begin{prop}\label{prop:comp dual k n}
For a complete subgraph $K_n\leq K$, the complementary dual $\hK_n$ is a $(d-n)$-pseudomanifold. 
\end{prop}
\begin{proof}
%For $n=1$, by Definition \ref{def:pman} it automatically follows that the $\hat{K_1}$ is a $(d-1)$-pseudomanifold.
We proceed by induction on indexing $n$. From the induction step
$$\hat{K}_{n-1}=\bigcap_{y\in K_{n-1}}L(y)$$
is a $(d-n+1)$-pseudomanifold. Now we prove for $n$. Write $K_n=\{v_1,v_2,\cdots,v_n\}$. Fix $v=v_n$ and write $K_{n-1}-v$. Following Definition \ref{def:compl dual} we have
\begin{align*}
    \hat{K}_{n}=\bigcap_{y\in K_n}L(y)=\bigcap_{y\in K_{n-1}}L(y)\cap L(v)=\hat{K}_{n-1}\cap L(v).
\end{align*}
 Then we must verify that $v\in\hat{K}_{n-1}$. For every $y\in K_{n-1}$, $v$ is adjacent to $y$ for $v\neq y$ since $K_n$ is a complete graph. Hence for every $y\in K_{n-1}$ we obtain $v\in L(y)$, that is, $v\in\hat{K}_{n-1}$.

 Now consider $L_{\hat{K}_{n-1}}(v)$ the unit link $L(v)$ within the $\hat{K}_{n-1}$. Since $\hat{K}_{n-1}$ is an induced subgraph of $K$,
 $$L_{\hat{K}_{n-1}}(v)=\{u\in\hat{K}_{n-1} ~|~ u\text{ adjacent to } v\in G\}=\hat{K}_{n-1}\cap L(v)=\hat{K}_{n-1}.$$
 As $\hK_{n-1}$ is a $(d-n+1)$-pseudomanifold by the induction hypothesis, the unit link of any vertex in $\hat{K}_{n-1}$, the $L_{\hat{K}_{n-1}}$ is a $((d-n+1)-1)=(d-n)$-pseudomanifold. This shows that $\hat{K}_{n}$ is a $(d-n)$-pseudomanifold, which proves the statement.
\end{proof}
\begin{prop}
For any subgraph $H\leq K$, the complementary dual $\hH$ is either the empty graph, a subpseudomanifold, a complete graph, a pyramid, or the entire pseudomanifold $K$.
\end{prop}
\begin{proof}
%For $k=0$ we have $H=\emptyset$, and $\hH=G$. For $k=1$ we then have $H=v$, which is a single vertex, and $\hH=L(v)$ is a $(d-1)$-pseudomanifold by Definition \ref{def:pman}.
We proceed by induction on the cardinality of the vertices $|V(H)|=k$. By the induction step the statement holds for all subgraphs with $|V(H)|\leq k$. Now we prove for $k+1$. Let $H$ be a subgraph with $|V(H)|=k+2\geq 2$. Fix $v\in H$ and write $H'=H-v$. Then
$$\hH=\hH'\cap L(v).$$
By the induction step we have $\hH'$ is one of the following cases.

Case $1$: $\hH'=\emptyset\text{ or } \hH'=K$. If $\hH'=\emptyset$, then $\hH=\emptyset$. If $\hH'=K$, then $\hH=L(v)$ is essentially a $(d-1)$-pseudomanifold by Definition \ref{def:pman}

Case $2$: $\hH'=K'$ is a subpseudomanifold. Hence we have $\hH=K'\cap L(v)$. Since $K'$ and $L(v)$ are subpseudomanifolds, $K'\cap L(v)$ is an induced subgraph of both. Following the local property of pseudomanifolds, the $K'\cap L(v)$ is either the empty graph, a subpseudomanifold, a complete subgraph (if the intersection is a clique), or a pyramid (if the intersection is a cone over a subpseudomanifold). 

Case $3$: $\hH'=K_n$ is a complete graph. Thus $\hH=K_n\cap L(v)$. Since any induced subgraph of a complete graph $K_n$ is complete, the $\hH$ is either $\emptyset$ or a complete subgraph.

Case $4$: $\hH'$ is a pyramid $P=K'+w$ where $K'$ is subpseudomanifold. Then 
$$\hH=P\cap L(v)=(K'+w)\cap L(v).$$ 
If $v=w$, then $\hH=K'\cap L(v)$ reduces to Case $2$. Moreover if $v\neq w$ and $v$ is adjacent to $w$, then $w\in L(v)$ and $\hH$ may retain the pyramid structure or be a complete graph. If $v\neq w$ and $v$ is not adjacent to $w$, then $w\notin L(v)$ and $\hH=K'\cap L(v)$ is again falling into Case $2$.

With all these demonstrations, we then complete the proof.
 \end{proof} 
 We now produce necessary elements, on the duality of $d$-pseudomanifold $K$, for proving the chromatic number $X(K)$ in Theorem \ref{thm:color pman}.  Notice that the dual $K^*$ of a $d$-pseudomanifold $K$ is $(d + 1)$-regular, since every simplex is connected with exactly $d + 1$ neighbors.
\begin{prop}\label{prop:daul g tri-free}
   For every $d$-pseudomanifold $K$, the dual graph $K^*$ is triangle-free. 
\end{prop}
\begin{proof}
Let $K^*$ contain a triangle formed by three maximal $d$-simplices $\sigma_1,\sigma_2,\sigma_3$, where each pair $\sigma_i\cap\sigma_j=\tau_{ij}$ shares a $(d-1)$-simplex. These faces are distinct as every $(d-1)$-simplex in $K$ is contained in exactly two maximal $d$-simplices, due to the $(d+1)$-regularity of $K^*$.

Now consider $\sigma=\sigma_1\cap\sigma_2\cap\sigma_3$. Since $\tau_{12},\tau_{13}$ are distinct $(d-1)$-faces of $\sigma_1$, the intersection $\sigma$ is a $(d-2)$-simplex. From Definition \ref{def:pman} we have $L(\sigma)$, the link of the $(d-2)$-simplex in $K$, is a $1$-pseudomanifold $C_n, n\geq 4$. However the $\sigma_1,\sigma_2,\sigma_3$ yield three distinct vertices, namely, $a,b,c\notin \sigma$ such that
$$\sigma_1=\sigma\cup\{a,b\},\sigma_2=\sigma\cup\{a,c\},\sigma_3=\sigma\cup\{b,c\}.$$

In $L(\sigma)$ the vertices $a,b,c$ are pairwise adjacent since $\sigma\cup\{a,b\},\sigma\cup\{a,c\},\sigma\cup\{b,c\}$ are $d$-simplices. Hence $L(\sigma)$ contains a triangle $K_3$, contradicting the fact that it must be $C_n,n\geq 4$. Therefore it follows that $K^*$ cannot contain a triangle. 
\end{proof}
\begin{prop}
For every $d$-pseudomanifold $K$, the dual graph $K^*$ can be cut into trees.
\end{prop}
\begin{proof}
 Following Proposition \ref{prop:daul g tri-free} we can break each cycle into a spanning tree of its vertices as follows. For a connected $K$, the  $K^*$ is connected. Then we can remove all edges not belonging to a spanning tree of $K^*$, and this results in a single tree. Furthermore in case $K$ is disconnected, the $K^*$ is disconnected. Hence the same terminology applied to each component yields a spanning forest, which is a family of trees, one for each component.
\end{proof}
Now we are ready to prove our main result. We demonstrate its proof on the lower bounds by observing the existence of the maximal simplices in the pseudomanifold $K$, and on the upper bounds by decomposing the dual graph $K^*$.
\begin{theorem}\label{thm:color pman}
The chromatic number $X(K)$ of any $d$-pseudomanifold $K$ satisfies
   $$d + 1 \leq X(K) \leq 2d + 2.$$ 
\end{theorem}
\begin{proof}
%In case $d=1$ we have $G=C_n, n\geq 4$, and the chromatic number 
 %$$X(K)=
 %\begin{cases}
    %2, & n \text{ is even}\\
    %3, & n\text{ is odd}
%\end{cases}$$which satisfies the bounds.
Assume that $K$ is a connected $d$-pseudomanifold. We then proceed by induction on $d$.  In the induction step we have
$$d\leq X(K) \leq 2d$$
for a $(d-1)$-pseudomanifold $K$. Now we prove for $d$. 
%For the lower bound $d+1$, it is clear by Definition \ref{def:pman}. Indeed we have a $d$-pseudomanifold $K$ that contains at least one $d$-simplex as a facet. Since the $d$-simplex is a complete graph on $d+1$, the clique number $c(G)\geq 1$. As we have $X(K)\geq c(G)$, the chromatic number $X(K)\geq d+1$.
For the bounds $d+1$, it is clear by Definition \ref{def:pman}. For the $2d+2$ we prove it as follows.

Step $1$: Decomposing $K^*$. First notice that $K^*$ is connected and $(d+1)$-regular. Following Proposition \ref{prop:daul g tri-free} we cut $K^*$ into two disjoint sets of spanning forests (trees), say, $F_1$ and $F_2$.

Step $2$: Coloring $F_1,F_2$. For all simplices within $F_1\in K^*$, color it by using the first batch of $d+1$ colors
$$C_1=\{1,2,\cdots,d+1\}.$$
For simplices within $F_2\in K^*$, color it with the second different batch of $d+1$ colors
$$C_2=\{d+2,d+3,\cdots,2d+2\}.$$

Step $3$: Combining $F_1\cup F_2$ and $C_1\cup C_2$. Since every simplex of $K$ belongs to one of the two forests, $F_1,F_2$, every facet is a vertex in both $F_1$ and $F_2$. Therefore every vertex of $K$ gets two potential color assignments, one from $F_1$ with a color in $C_1$ and one from $F_2$ with a color in $C_2$. Combining the two forests $F_1,F_2$ along with their colors $C_1, C_2$, it follows that every vertex gets a single color from the total set coloring of size $2d+2$, more precisely,
$$C_2\cup C_2=\{1,2,\cdots,d+1,d+2,\cdots,2d+1,2d+2\}.$$
This completes the claim.
 \end{proof}
\begin{cor}
If $K'\leq K$ is a subpseudomanifold of a pseudomanifold $K$, then we have $X(K') \leq X(K)$. In particular, 
$$\text{dim} K'+1\leq X(K')\leq 2\text{dim} K'+2.$$
\end{cor}
\begin{proof}
Immediately follows from Theorem \ref{thm:color pman}.
\end{proof}
\begin{rmk}
Oberserving on Theorem \ref{thm:color pman} and \cite[Theorem 1]{Kni21}, we see that while the $d$-manifolds and the broader class of $d$-pseudomanifolds share the general chromatic bounds 
$$d+1\leq X(M), X(K)\leq 2d+2,$$ the relationship between their effective ranges remains an intriguing question. As in \cite[\S1.1]{Kni21}, for the discrete manifolds $M$ there is evidence supporting a conjectured tighter upper bound of $\lceil3(d+1)/2\rceil$ where $\lceil-\rceil$ is a ceiling function, and constructions exist that reach chromatic numbers scaling as $1.5d$. Whether the more flexible local structure of pseudomanifolds permits chromatic numbers exceeding this conjectured manifold ceiling, or conversely, if it imposes even stricter limits, is a subject for further investigation in \S\ref{sec:sharper bounds}.
\end{rmk}

%\textcolor{red}{\begin{remark}\begin{itemize}\item There is a known result related to the Hamiltonicity, see Bruno Benedetti-Marta Pavelka: Higher-dimensional counterexamples to Hamiltonicity. We do not use this anywhere in the paper, so we omit it. \item In the proof of Proposition \cite{prop:comp dual k n} we omit the convention case: First notice that for $n=0$, we have $K_0=\emptyset$. By convention, the intersection over an empty graph set is $K$. Since $K$ is a $d$-pseudomanifold, we have $\hat{K}=G$ is a $d$-pseudomanifold.\end{itemize} \end{remark}}
\section{Chromatic Arithmetic on Pseudomanifolds}\label{sec:chrom arith}
In this section we take a look at the graph operations and their chromatic properties on the pseudomanifolds. The chromatic numbers are proved for general graphs in \cite[\S3]{Kni21}. However something is interesting here: The Zykov join in Definition \ref{def:zykov join} and the Cartesian simplex product in Definition \ref{def:cartesian pro} preserve pseudomanifolds, graphs that recursively have the property that all unit links are again one lower-dimensional pseudomanifolds. 
\begin{prop}
Let $K$ be an $m$-pseudomanifold and $K'$ an $n$-pseudomanifold. Then the Zykov join $K+K'$ is an $(m+n+1)$-pseudomanifold.
\end{prop} 
\begin{proof}
We proceed by induction on the total dimension $m+n$. Inductively assume that the statement holds for all the Zykov joins where the sum of the dimensions is $<m+n$. We now consider $K$ with dimension $m$ and $K'$ with dimension $n$, and $m+n\geq 0$.

Suppose $v$ is a vertex in $K+K'$. Then there are two possibilities, whether $v$ is originated from $K$ or $K'$, that is, 
$$L_{K+K'}(v)=L_K(v)+K' \text{ or } L_{K+K'}(v)=K+L_{K'}(v).$$
By Definition \ref{def:pman} we have both $L_K(v)$ and $L_{K'}(v)$ are $(m-1)$-pseudomanifold and $(n-1)$-pseudomanifold, respectively. From the induction hypothesis the Zykov join of an $(m-1)$-pseudomanifold and an $n$-pseudomanifold, and that of an $m$-pseudomanifold and an $(n-1)$-pseudomanifold, are $(m+n)$-pseudomanifold. Finally we observe that every unit link is an $(m+n)$-pseudomanifold, so that with Definition \ref{def:pman} it follows that $K+K'$ is an $(m+n+1)$-pseudomanifold. We then prove the claim. 
\end{proof}
Notice that for two general graphs $G$ and $H$, there is a formula \cite[Lemma 2]{KE21}
$$\text{dim}(G +H)= \text{dim} G +\text{dim} H +1.$$
Furthermore by \cite[Lemma 1]{Kni21}, the chromatic number $X(G+H)$ is additive 
\begin{myeq}\label{add chrom no}
   X(G+H)= X(G)+X(H).
\end{myeq}
In particular, 
\begin{myeq}\label{eq:sus gr}
  X(G+S^0)=X(G)+1
\end{myeq}
where $S^0=\{\{a,b\},\emptyset\}$. In the topological setting, the $G+S^0=SG$ is a suspension of $G$.

There was a conjecture on sphere coloring, which stated that any $k$-sphere $S^k$ has chromatic number $X(S^k)=k+1$ or $X(S^k)=k+2$ in \cite[\S8.3]{Kni18}. Later in \cite[Corollary 1]{Kni21} it was disproved by a counterexample: Any $k$-sphere $S^k$ has $X(S^k)= 3k$ or $X(S^k)=3k+1$. This shows that the $X(S^k)$ can be relatively large. As a complement to this, there is also a chromatic number $X(S^k)$ which is relatively small, described as follows.
\begin{prop}\label{prop:color sphere min}
For any $k$-sphere $S^k$, there is chromatic number $$X(S^k)=2\lceil (k+1)/2\rceil$$
where $\lceil - \rceil$ is a ceiling function.
\end{prop} 
\begin{proof}
Consider a $(2k-1)$-sphere $S^{2k-1}$ which is an even-cycle sphere,
$$C_n + C_n + \dots + C_n$$
such that $n\geq 4$ is even.
Observe that the $S^{2k}=S^{2k-1}+S^0$. Therefore the former $S^{2k-1}$ uses $2k$ colors, while the later $S^{2k}$ uses $2k+1$ by \eqref{eq:sus gr}. This proves the statement.
\end{proof}

Following the construction of the $k$-sphere $S^k$ in the proof of Proposition \ref{prop:color sphere min} we have that any $k$-sphere is constructed from two types of cycles $C_n$ for $n\geq 4$. In symbols, any sphere can be built by the Zykov joins, $n\geq 4$,
$$C_n+C_n+\cdots+C_n.$$
If every $n\geq 4$ is even, we call $S^k$ an even-cycle sphere, and if every $n\geq 5$ is odd, then $S^k$ is an odd-cycle sphere.
To be used later, we record the chromatic numbers
\begin{myeq}\label{eq:color spheres}
    X(S^k)=2\lceil (k+1)/2\rceil \text{ or } X(S^k)=3\lceil (k+1)/2\rceil
\end{myeq}where the first factor is for the even-cycle sphere, and the second for the odd-cycle sphere. 

Using Definition \ref{def:cartesian pro} we also record the following result. 
\begin{prop}
Let $K$ be an $m$-pseudomanifold and $K'$ an $n$-pseudomanifold, with $m, n\geq 1$. Then the Cartesian simplex product $(K\times K')_1$ is an $(m+n+1)$-pseudomanifold. 
\end{prop}
\begin{proof}
Following the same proof as in \cite[Lemma 3]{Kni21}, or more detail in \cite[Lemma 5]{Kni23}.
\end{proof}
Recalling the construction in \cite[\S5.1]{Kni21}. Let $x = (x_0, \dots,x_{d-2})$ be a $(d-2)$-simplex in a $d$-pseudomanifold $K$. A dual link of the simplex $x$ is defined to be 
$$O(K)=L(x_0)\cap \dots \cap L(x_{d-1})$$ 
where the notation $O(K)$ is determined as the odd length by Fisk in \cite[VI.2]{Fis77}. We then call this odd structure $O(K)$ the Fisk set if $K$ is a graph, and the Fisk variety if $K$ is a pseudomanifold.
\begin{prop}
For an $m$-pseudomanifold $K$ and an $n$-pseudomanifold $K'$ with $m, n \geq 2$, 
$$O(K+K') = K+O(K')\cup O(K)+K'.$$
\end{prop}
\begin{proof}
Following the same proof as in \cite[Lemma 6]{Kni21}.
\end{proof}

\section{The Limit of Chromatic Pseudomanifolds}\label{sec:sharper bounds}
The purpose of this section is to look at the sharpness of the previous bounds $2d+2$ in Theorem \ref{thm:color pman}. For the sharper bounds we follow the two conjectures in \cite[\S 1.1]{Kni21} and interpret them in the context of pseudomanifolds. The first version of the conjectures is more practical, while the second is stricter. 
\begin{conj}\label{conj:bound 2d+1}
  For any $d$-pseudomanifold $K$, the chromatic number $X(K)$ satisfies $$d+1\leq X(K)\leq 2d+1.$$
\end{conj}
\begin{conj}\label{conj:bound ceiling 1.5(d+1)}
  For any $d$-pseudomanifold $K$, the chromatic number $X(K)$ satisfies $$d+1\leq X(K)\leq\lceil 3(d+1)/2 \rceil$$
  where $\lceil - \rceil$ is a ceiling function.
\end{conj}

 Responding to Conjectures \ref{conj:bound 2d+1} and \ref{conj:bound ceiling 1.5(d+1)}, we record the following special cases.
\begin{theorem}\label{thm:bound 2d+1}
If a $d$-pseudomanifold $K=S^k+K'$ such that $S^k$ is a $k$-sphere and $K'$ is a $(d-k-1)$-pseudomanifold, $$d+1\leq X(K)\leq 2d+1.$$  
\end{theorem}
\begin{proof}
Rewriting the identity \eqref{add chrom no}
$$X(K)=X(S^k)+X(K').$$
Consider the maximum bounds
\begin{myeq}\label{eq:max sphere pseudoman}
    X(S^k)\leq 3k+1 \text{ and } X(K')\leq 2(d-k)
\end{myeq}
where the first term is in \eqref{eq:color spheres} and the second by Theorem \ref{thm:color pman}. There are two possibilities.

Case $1$: $k$ is even $(k=2p)$. Then $X(S^k)=X(S^{2p})\leq 3p$. Using \eqref{add chrom no} and \eqref{eq:max sphere pseudoman},
$$X(K)=X(S^{2p})+X(K')\leq 2d-p.$$
It follows that $X(K)\leq 2d+1$ is always true for $p\geq 0$.

Case 2: $k$ is odd $(k=2p+1)$. Again following \eqref{add chrom no} and \eqref{eq:max sphere pseudoman},
$$X(K)=X(S^{2p+1})+X(K')\leq 2d-p+1.$$
Thus $X(K)\leq 2d+1$ is always true for $p\geq 0$. This completes the proof. 
\end{proof}

As shown in Proposition \ref{prop:color sphere min}, the chromatic number of an even-cycle sphere $S^k$ can be relatively small. Following Theorem \ref{thm:bound 2d+1} it gives rise to the following studies.

\textbf{Case A: $S^k$ is an even-cycle sphere.} By \eqref{eq:color spheres} and Theorem \ref{thm:bound 2d+1},
\begin{myeq}
   X(K)=X(S^k)+X(K')\leq 2d-k.
\end{myeq}
We then require $$X(K)\leq\lceil3(d+1)/2\rceil.$$ However this occurs if and only if 
$$k\geq 2d-\lceil(3d+1)/2\rceil$$
which is equivalent to saying that the dimension $k$ is close to $d$, more precisely, 
\begin{myeq}\label{eq: k close d}
    k=
\begin{cases}
    d/2-1, & d \text{ is even}\\
    (d+1)/2-1, & d\text{ is odd.}
\end{cases}
\end{myeq}

\textbf{Case B: $S^k$ is an odd-cycle sphere.} In this case the upper bound $$X(K)\leq\lceil(3d+1)/2\rceil$$ may no longer true for $S^k$ as an odd-cycle sphere by our observation in Table \ref{tab:counterex-ceil}. In the table we choose the maximum chromatic number $$X(K')_{\text{max}}=2(d-k)$$ for the subpseudomanifold $K'$. This provides a conservative estimate that accounts for the worst-case chromatic number, ensuring the subsequent inequalities hold in general.

We formulate the above analysis of \textbf{Cases A, B} in the following statement.
\begin{table}[ht]
    \centering
    \begin{tabular}{@{}cccccccc@{}}
        \toprule
        $d$ & $k$ & $S^k$ & $\text{dim} K'$ & $X(S^k)$ & $X(K')_{\text{max}}$ & $X(K)_{\text{max}}$ & X(K) \\
        \midrule
        3 & 1 & $C_5$           & 1 & 3  & 3 & 6 & 6\\
        5 & 1 & $C_5$           & 3 & 3  & 7 & 10 & 9\\
        5 & 3 & $C_5 + C_5$     & 1 & 6  & 3 & 9 & 9\\
        7 & 1 & $C_5$           & 5 & 3  & 11 & 14 & 12\\
        7 & 3 & $C_5 + C_5$     & 3 & 6  & 7 & 13 & 12\\
        7 & 5 & $C_5 + C_5+C_5$ & 1 & 9 & 3 & 12 & 12\\
        \bottomrule
    \end{tabular}
    \caption{We write notations $X(K)_{\text{max}}=X(S^k)+X(K')_{\text{max}}, X(K')_{\text{max}}=2(d-k),\text{ and } X(K)=\lceil3(d+1)/2\rceil$.}
    \label{tab:counterex-ceil}
\end{table}
 
\begin{theorem}\label{thm:bound ceiling 1.5(d+1)}
If a $d$-pseudomanifold $K=S^k+K'$ such that $S^k$ is an even-cycle $k$-sphere, $K'$ is a $(d-k-1)$-pseudomanifold, and $k$ is close to $d$ as in \eqref{eq: k close d}, $$d+1\leq X(K)\leq \lceil 3(d+1)/2\rceil$$
where $\lceil - \rceil$ is a ceiling function.
\end{theorem}

\bibliographystyle{siam}
\bibliography{RECPseudoman}
\end{document}